\documentclass{amsart}

\usepackage{amsthm, amsfonts, amssymb, amsmath, latexsym, enumerate, times}
\usepackage[latin1]{inputenc}
\usepackage{mathrsfs}
\usepackage{array}
\usepackage{comment}
\usepackage{paralist}
\usepackage{color}

\newtheorem{theorem}{Theorem}
\newtheorem{lemma}[theorem]{Lemma}

\newtheorem{proposition}[theorem]{Proposition}

\theoremstyle{definition}

\newtheorem{remark}[theorem]{Remark}
\newtheorem{example}[theorem]{Example}

\renewcommand{\O}{{\mathcal O}}

\newcommand{\Proj}{{\mathbb P}}

\newcommand{\G}{{\mathbb G}}

\def\geq{\geqslant}
\def\leq{\leqslant}

\begin{document}
 
\title[Maximum number of complexes containing subvarieties of Grassmannians]{On the maximum number of complexes of a given degree containing subvarieties of Grassmannians}

\author{Ciro Ciliberto}
\address{Dipartimento di Matematica, Universit\`a di Roma Tor Vergata, Via O. Raimondo
 00173 Roma, Italia}
\email{cilibert@axp.mat.uniroma2.it}
 
 
 
\maketitle

\medskip

\begin{abstract}  Let $\G(k,r)$ be the Grassmannian of $k$--subspaces in $\Proj^r$ embedded in $\Proj^{N(k,r)}$, with $N(k,r)={{r+1}\choose {k+1}}-1$, via the Pl\"ucker embedding.  In this paper, extending some classical results by Gallarati (see \cite {Gal1,Gal2}), we give a sharp upper bound for the number of independent sections of $H^0(\G(k,r), \O_{\G(k,r)}(m))$ vanishing on a subvariety $X$ of $\G(k,r)$ such that the union of the $k$--subspaces corresponding to the points of $X$ spans $\Proj^r$.
\end{abstract}

\section*{Introduction} Let $\G(k,r)$ be the Grassmannian of $k$--subspaces (i.e., linear subspaces of dimension $k$) in $\Proj^r$ embedded in $\Proj^{N(k,r)}$, with $N(k,r)={{r+1}\choose {k+1}}-1$, via the Pl\"ucker embedding. To avoid trivial cases we will assume $k\neq 0,r-1$. 

For every positive integer $m$ we  consider $H^0(\G(k,r), \O_{\G(k,r)}(m))$ whose dimension we denote by
$\varepsilon_{k,r}(m)$. This is a well known number, computed for instance in \cite [Thm. III, p. 387]{HP}, which is not necessary to make explicit here. Given any non--zero section $s\in H^0(\G(k,r), \O_{\G(k,r)}(m))$, the zero locus scheme $(s)$ of $s$ is called a \emph{$m$--complex} of $\G(k,r)$.

Let now $X$ be an irreducible, projective subvariety of $\G(k,r)$. We can consider the subvariety
$$
Z(X)=\bigcup_{\pi\in X}\pi
$$
of $\Proj^r$. We will say that $X$ is \emph{Grassmann non--degenerate} if $Z(X)$ is non--degenerate in $\Proj^r$, i.e., if $Z(X)$ spans $\Proj^r$.  We will say that $X$ presents the \emph{cone case}  if the subspaces corresponding to the points of $X$ pass through one and the same linear subspace of dimension $k-1$, in which case 
$Z(X)$ is a cone with vertex a linear space of dimension $k-1$. 

Let $\mathcal I_{X,\G(k,r)}$ be the ideal sheaf of $X$ in $\G(k,r)$. We will set 
$$
\theta_X(m):=h^0(\G(k,r),\mathcal I_{X,\G(k,r)}(m))
$$
that is the maximum number of independent $m$--complexes of $\G(k,r)$ containing $X$.

In the two papers \cite {Gal1,Gal2}, D. Gallarati proved the following two results:

\begin{theorem}\label{thm:Gal1} Let $X$ be an irreducible projective  Grassmann non--degenerate curve in $\mathbb G(1,r)$. Then for any positive integer $m$ one has
\begin{equation}\label{eq:Gal1}
\theta_X(m) \leq \varepsilon_{1,r}(m)-m(r-1)-1.
\end{equation}
Moreover:\\
\begin{inparaenum}
\item [(i)] if $m>1$ and $X$ is a rational normal curve of degree $r-1$ (equivalently if $Z(X)$ is a rational normal scroll surface of degree $r-1$), then the equality holds in \eqref{eq:Gal1};\\
\item [(ii)]  if $m=1$ equality holds in \eqref{eq:Gal1} if and only if either $X$ is a rational normal curve of degree $r-1$ or $X$ presents the cone case. 
\end{inparaenum}
\end{theorem}

\begin{theorem}\label{thm:Gal2} Let $X$ be an irreducible projective  Grassmann non--degenerate curve in $\mathbb G(k,r)$. Then 
\begin{equation}\label{eq:Gal2}
\theta_X(1) \leq {{r+1}\choose {k+1}} -r+k-1=\varepsilon_{k,r}(1) -r+k-1.
\end{equation}
Moreover, equality holds in \eqref{eq:Gal2} if and only if either $X$ is a rational normal curve of degree $r-k$ (equivalently $Z(X)$ is a $(k+1)$--dimensional rational normal scroll of degree $r-k$) or $X$ presents the cone case. 
\end{theorem}

In this note we will prove an extension of Gallarati's results. Before stating our  result, let us introduce some notation. Given positive integers $n,k,r,m$ we will set
$$
\theta(n,k,r,m)= \varepsilon_{k,r}(m)-(r-k) \sigma(n,m)- \tau(n,m)
$$
where $\sigma(1,m)=m$, $\sigma(2,m)=\frac {m(m+1)}2$,  whereas if $n\geq 3$ one has
$$
\sigma(n,m)=\sum_{i_{n-2}=1}^m\sum_{i_{n-3}=1}^{i_{n-2}} \cdots \sum_{i_1=1}^{i_2} \frac {i_1(i_1+1)}2,
$$
and 
$$\tau(1,m)=1,\quad  \tau(2,m)=m+1,\quad   \tau(3,m)=\frac {m(m+1)}2+m+1$$ 
whereas if $n\geq 4$ one has
$$
\begin{aligned}
&\tau(n,m)=\sum_{i_{n-3}=1}^m\sum_{i_{n-4}=1}^{i_{n-3}} \cdots \sum_{i_1=1}^{i_2} \frac {i_1(i_1+1)}2+\cr
&+ \sum_{i_{n-4}=1}^m\sum_{i_{n-5}=1}^{i_{n-4}} \cdots \sum_{i_1=1}^{i_2} \frac {i_1(i_1+1)}2+\cdots +  \frac {m(m+1)}2+m+1.
\end{aligned}
$$

One notes that $\theta(1,1,r,m)$ coincides with Gallarati's upper bound in \eqref {eq:Gal1} and 
$\theta(1,k,r,1)$ coincides with the upper bound in  \eqref {eq:Gal2}.
In addition we notice that
\begin{equation}\label{eq:triv}
\sigma(n,m)=\sum_{i=1}^m\sigma(n-1,i) \quad \mbox {and}\quad  \tau(n,m)=\sum_{i=1}^m\tau(n-1,i)+1.
\end{equation}

Moreover, given positive integers $n>1,k,r,m$ we set
$$
\bar\theta(n,k,r,m)= \varepsilon_{k,r}(m)-{{m+n-1}\choose n}(r-k-n)-{{m+n}\choose n}.
$$

Then we can state our  result:

\begin{theorem}\label{thm:main} Let $X$ be an irreducible projective Grassmann non--degenerate variety of dimension $n$  in $\mathbb G(k,r)$. 

(a) If either $n=1$ or $n>1$ and $X$  does not present the cone case, then 
\begin{equation}\label{eq:Gal3}
\theta_X(m) \leq \theta(n,k,r,m).
\end{equation}
Moreover:\\
\begin{inparaenum}
\item [(i)] for $m=1$ equality holds in \eqref{eq:Gal3} if and only if either $X$ is a rational normal 
variety of degree $r-k$ in a subspace of dimension $r-k+n-1$ in $\Proj^{N(k,r)}$ or $n=1$ and $X$ presents the cone case;\\
\item [(ii)] for $n=1$ and $m>1$ equality holds in \eqref{eq:Gal3} if $X$ is a rational normal 
curve of degree $r-k$;\\
\item [(iii)] for  $n>1$ and $m>1$, equality holds in \eqref{eq:Gal3} if and only if $X$ is a rational normal 
variety of degree $r-k$ spanning a subspace of dimension $r-k+n-1$ in $\Proj^{N(k,r)}$. \end{inparaenum}

(b) If $n>1$ and $X$ presents the cone case then 
\begin{equation}\label{eq:Gal4}
\theta_X(m) \leq \bar \theta(n,k,r,m).
\end{equation}
Moreover:\\
\begin{inparaenum}
\item [(iv)] equality holds in \eqref {eq:Gal4} for $m=1$;\\
\item [(v)] if equality holds in \eqref {eq:Gal4} for some $m\geq 2$, then $X$ is a  rational normal 
variety of degree $r-k-n+1$ in a linear subspace of dimension $r-k$ contained in $\G(k,r)$
whose points correspond to $k$--subspaces containing a fixed $(k-1)$--subspace, or, equivalently, $Z(X)$ is a cone with vertex a $(k-1)$--subspace over a variety of dimension $n$ and minimal degree $r-k-n+1$ in a $(r-k)$--linear subspace of $\Proj^r$. 
\end{inparaenum}
\end{theorem} 

The proof of part (a) works by induction on the dimension $n$ of $X$.  We first treat the curve case in \S \ref {sec:curve}, then the surface and threefold case in \S \ref {sec:gen}. Then, to complete the proof of part (a) in the rest of \S \ref {sec:gen}, we proceed by induction using a classical method going back to Castelnuovo, passing from $X$ to a general hyperplane section of it. We will treat case (b) in \S \ref {sec:cone}, again using Castelnuovo's method. In \S \ref {sec:ex} we finish with some examples showing that the varieties mentioned in the statement of Theorem \ref{thm:main}, for which equality holds in \eqref {eq:Gal3} and \eqref {eq:Gal4}, do really occur. \medskip

{\bf Acknowledgements:} The author is a member of GNSAGA of INdAM.

\section{The curve case}\label{sec:curve}

In this section we will prove Theorem \ref {thm:main} in the curve case $n=1$. A great part of this proof is already in Gallarati's papers \cite {Gal1,Gal2}. However it is the case to review some details of the proof that are a bit too terse in Gallarati's treatment.

Let us start with the following  easy lemma:

\begin{lemma}\label{lem:basic} Let $X$ be an irreducible projective  Grassmann non--degenerate curve in $\mathbb G(k,r)$. Then given $r-k$ general points $P_1,\ldots, P_{r-k}$ of $X$, there is a hyperplane section of $\mathbb G(k,r)$ not containing $X$ and containing $P_1,\ldots, P_{r-k}$.
\end{lemma}

\begin{proof} Given a linear subspace $\Pi$ of dimension $r-k-1$, one can consider the Schubert cycle of all linear $k$--subspaces of $\Proj^r$ intersecting $\Pi$. This is a linear complex, i.e.,  a hyperplane section $H_\Pi$ of $\G(k,r)$. Given $P_1,\ldots, P_{r-k}$ general points of $X$, these correspond to $r-k$ general subspaces sweeping out $Z(X)$. Let us fix a general point $p_i\in P_i$ for any $1\leq i\leq r-k$, so that $p_1,\ldots, p_{r-k}$ are general points of $Z(X)$. The Grassmann non--degeneracy of 
$X$ is equivalent to $Z(X)$ being non--degenerate, so $p_1,\ldots, p_{r-k}$ are linearly independent in $\Proj^r$. Then $\Pi=\langle p_1,\ldots, p_{r-k}\rangle$ is a $(r-k-1)$--subspace of $\Proj^r$, and actually it is a general such subspace, so that it intersects $Z(X)$, that has dimension $k+1$, in finitely many points. Then $\Pi$ intersects the subspaces $P_1,\ldots, P_{r-k}$ but does not intersect all the subspaces $P\in X$. Hence $H_\Pi$ is a hyperplane section of $\mathbb G(k,r)$ not containing $X$ and containing $P_1,\ldots, P_{r-k}$, as required. \end{proof}

Let $X\subset \Proj^r$ be an irreducible non--degenerate projective variety. For all non--negative integers $m$ we let $h_X(m)$ be the \emph{Hilbert function} of $X$, i.e., $h_X(m)$ is  the dimension of the image of the restriction map
$$
H^0(\Proj^r, \mathcal O_{\Proj^r}(m))\longrightarrow H^0(X, \O_X(m)).
$$ 
So $h_X(m)-1$ is the dimension of the linear series cut out on $X$ by the hypersurfaces of degree $m$ of $\Proj^r$.

Now suppose that $X$ is an irreducible projective Grassmann non--degenerate variety of dimension $n$  in $\mathbb G(k,r)$. For any non--negative integer $m$, we have the exact sequence
$$
0\longrightarrow \mathcal I_{X,\G(k,r)}(m)\longrightarrow \O_{\G(k,r)}(m)\longrightarrow  \O_{X}(m)\longrightarrow 0
$$
which gives
\begin{equation}\label{eq:seq}
0\longrightarrow H^0(\G(k,r),\mathcal I_{X,\G(k,r)}(m))\longrightarrow H^0(\G(k,r),\O_{\G(k,r)}(m))\longrightarrow  H^0(X,\O_{X}(m))
\end{equation}
The image of 
\begin{equation}\label{eq:rho}
\rho_m: H^0(\G(k,r),\O_{\G(k,r)}(m))\longrightarrow  H^0(X,\O_{X}(m))
\end{equation}
coincides with the image of 
$$H^0(\Proj^r,\O_{\Proj^r}(m))\longrightarrow  H^0(X,\O_{X}(m))$$
because $\G(k,r)$ is projectively normal and therefore
\begin{equation}\label{eq:rho1}
H^0(\Proj^r,\O_{\Proj^r}(m))\longrightarrow  H^0(\G(k,r),\O_{\G(k,r)}(m))
\end{equation}
is surjective for all non--negative integers $m$. So the dimension of the image of $\rho_m$ is $h_X(m)$. 

 By \eqref {eq:seq} we have
\begin{equation}\label{eq:sequ}
\theta_X(m) =\varepsilon _{k,r}(m)-h_X(m) 
\end{equation}

Now we are ready for the:

	\begin{proof}[Proof of Theorem \ref {thm:main} for $n=1$] By \eqref {eq:sequ}, to bound $\theta_X(m)$ from above we have to bound $h_X(m)$ from below. To do so, notice that given $m(r-k)$ general points of $X$, there is some $m$--complex of $\G(k,r)$ containing those points and not containing $X$. It suffices to divide the $m(r-k)$ general points of $X$ in $m$ subsets of $r-k$ general points, and then apply Lemma \ref {lem:basic}, thus finding the required $m$--complex  splitting in $m$ linear complexes each  containing one of the $m$ subsets of $r-k$  points. This proves that 
\begin{equation}\label{eq:sequor}
h_X(m)-1\geq m(r-k). 
\end{equation} 
Then by \eqref {eq:sequ} we get
$$
\theta_X(m) =\varepsilon _{k,r}(m)-h_X(m) \leq \varepsilon _{k,r}(m)-m(r-k)-1=\theta(1,k,r,m)
$$
proving  \eqref {eq:Gal3}  in this case. 

The proof of (i) in the case $n=1$ is contained in \cite {Gal1,Gal2} and we do not dwell on this here.
As for (ii), if $X$ is a rational normal curve of degree $r-k$, then for all positive integers $m$ we have $h_X(m)=m(r-k)+1$ and therefore from \eqref {eq:sequ} we find that $\theta_X(m)=\theta(1,k,r,m)$ as wanted. \end{proof}

\section{The general case}\label{sec:gen}

Before treating the general case it is necessary to work out first the surface and the threefold cases.
We start with some lemmata.

 \begin{lemma}\label{lem:basic0} Let $X$ be an irreducible projective  Grassmann non--degenerate variety in $\mathbb G(k,r)$ not presenting the cone case. Let $P_1,\ldots, P_{r-k}$ be general points in $X$, which correspond to  $k$--subspaces of $\Proj^r$ that we denote by the same symbols. Then 
  $$
 \langle P_1,\ldots, P_{r-k}\rangle =\Proj^r.
  $$
  \end{lemma}
  
  \begin{proof}  One has $\dim (\langle P_1, P_2\rangle)=k+a_1$, with $a_1\geq 1$. If $a_1=1$, two $k$--subspaces corresponding to general points of  $X$ span a linear space of dimension $k+1$, so they intersect in a linear subspace of dimension $k-1$. Then either they all lie in a linear space of dimension $k+1$ or they all pass through a linear space of dimension $k-1$, but neither case can occur by the non--degeneracy of $X$ and the fact that $X$ does not present the cone case. Hence $a_1\geq 2$. 
  
  Next, if $\langle P_1, P_2\rangle=\Proj^r$ the claim clearly holds. So suppose that $\langle P_1, P_2\rangle$ is a proper subspace of $\Proj^r$, hence $k+a_1<r$. Then $P_3$ cannot lie in $\langle P_1, P_2\rangle$  by the non--degeneracy of $X$, so $\dim (\langle P_1, P_2, P_3 \rangle)=k+a_1+a_2$, with $a_2\geq 1$. If $\langle P_1, P_2, P_3 \rangle=\Proj^r$, the claim is clearly true, otherwise we iterate the above argument. The upshot is that  either a proper subset of $P_1,\ldots, P_{r-k}$ spans $\Proj^r$ and we are done, or 
  $$
  \dim( \langle P_1,\ldots, P_{r-k}\rangle)=k+a_1+\cdots +a_{r-k-1}\geq k+r-k=r
  $$
 proving the lemma. \end{proof}
  
   \begin{lemma}\label{lem:basic1} Let $X$ be an irreducible projective  Grassmann non--degenerate variety in $\mathbb G(k,r)$. Then 
   $$
   h_X(1)\geq r-k+1.
   $$
  \end{lemma}
  
  \begin{proof} If $X$ is a curve, the assertion follows from \eqref {eq:sequor}. Assume that $X$ has dimension $n\geq 2$. Fix  $P_1,\ldots, P_{r-k}$ general points in $X$. Let $C$ be the  complete intersection curve of $X$ with $n-1$ general hypersurfaces of degree $d\gg 0$ containing $P_1,\ldots, P_{r-k}$. One clearly has $\theta_C(1)=\theta_X(1)$. By Lemma \ref {lem:basic0}, $C$ is Grassmann non--degenerate, and therefore, by \eqref {eq:Gal3} one has 
  $$\theta_X(1)=\theta_C(1)\leq \theta(1, k, r, 1)={{r+1}\choose {k+1}} -r+k-1,$$
whence the assertion immediately follows.
  \end{proof}

 \begin{lemma}\label{lem:basic} Let $X$ be an irreducible projective  Grassmann non--degenerate variety of dimension $n\geq 2$ in $\mathbb G(k,r)$ not presenting the cone case. If $Y$ is a general hyperplane section of $X$, then $Y$ is also Grassmann non--degenerate and does not present the cone case.
  \end{lemma}
  
  \begin{proof} Let $P_1,\ldots, P_{r-k}$ be general points in $X$.  By Lemma \ref {lem:basic1}, there is a hyperplane section $Y$ of $X$ containing $P_1,\ldots, P_{r-k}$, and $Y$ can be considered to be a general hyperplane section of $X$. By Lemma \ref {lem:basic0}, $Y$ is Grassmann non--degenerate. 
  
 To finish the proof we have to show that for a general hyperplane section $Y$ of $X$, $Y$ does not present the cone case. Suppose, by contradiction, that this is the case. Then given two general points $P_1,P_2\in X$, there is a hyperplane section of $X$ containing them, and therefore the $k$--subspaces $P_1,P_2$ intersect along a $(k-1)$--subspace. Then either all subspaces corresponding to points of $X$ pass through the same $(k-1)$--subspace or they all lie in the same $(k+1)$--subspace. Both possibilites lead to contradictions since $X$ is Grassmann non--degenerate and does not present the cone case.  \end{proof}
 
 Next let $X\subset \Proj^r$ be an irreducible non--degenerate projective variety. Let $Y$ be a general hyperplane section of $X$. For all $m\geq 1$  one has the  inequality 
\begin{equation}\label{eq:cas}
h_X(m)-h_X(m-1)\geq h_Y(m)
\end{equation}
and equality holds for all $m$ if and only if $X$ is projectively normal (see \cite [Lemma (3.1)]{Ha}).

\begin{proof}[Proof of part (a) of Theorem \ref {thm:main} for surfaces] Let $X$ be an irreducible projective Grassmann non--degenerate surface  in $\mathbb G(k,r)$ not presenting the cone case and let $Y$ be a general hyperplane section of $X$. By Lemma \ref {lem:basic}, $Y$ is Grassmann non--degenerate. By \eqref {eq:cas}, for all positive integers $m$ one has the inequalities
\begin{align}\label{eq:inequ}
&h_X(m)-h_X(m-1)\geq h_Y(m)\\
\nonumber&\ldots\\
\nonumber &h_X(1)-h_X(0)\geq h_Y(1)
\end{align}
and summing up we get
\begin{equation}\label{eq:inequ2}
h_X(n)\geq \sum_{i=1}^mh_Y(i)+h_X(0)=\sum_{i=1}^mh_Y(i)+1. 
\end{equation}
By \eqref {eq:sequ} we have
\begin{equation}\label{eq:inequ3}
\varepsilon _{k,r}(m)-\theta_X(m)=h_X(n)\geq \sum_{i=1}^mh_Y(i)+1=\sum_{i=1}^m(\varepsilon _{k,r}(i)-\theta_Y(i))+1
\end{equation}
whence, using Theorem \ref {thm:main} for the curve $Y$ (that we can do because $Y$ is Grassmann non--degenerate), we have
\begin{align}\label{eq:ineq}
\theta_X(m)&\leq \varepsilon _{k,r}(m)-\sum_{i=1}^m(\varepsilon _{k,r}(i)-\theta_Y(i))-1=\\
\nonumber &=\varepsilon _{k,r}(m)-\sum_{i=1}^m\varepsilon _{k,r}(i)+\sum_{i=1}^m\theta_Y(i)-1=\\
\nonumber &=\sum_{i=1}^m\theta_Y(i)-\sum_{i=1}^{m-1}\varepsilon _{k,r}(i)-1\leq \\
\nonumber  &\leq \sum_{i=1}^m \Big (\varepsilon _{k,r}(i)-i(r-k)-1 \Big)- \sum_{i=1}^{m-1}\varepsilon _{k,r}(i)-1=\\
\nonumber&=\varepsilon _{k,r}(m)-\frac {m(m+1)}2(r-k)-m-1=\theta(2,k,r,m)
\end{align}
proving \eqref {eq:Gal3} in this case. 

Next let us prove (i). If $\theta_X(1)=\theta(2,k,r,1)$ then the above argument shows that $\theta_Y(1)= \theta(1,k,r,1)$. By Theorem \ref {thm:main}(i) for $Y$ and by Lemma \ref {lem:basic},  $Y$ is a rational normal curve of degree $r-k$, because $Y$ cannot present the cone case. Hence $X$ is a rational normal surface of degree $r-k$. Conversely, if 
$X$ is a rational normal surface of degree $r-k$, then $h_X(1)=r-k+2$ and therefore
$$
\theta_X(1)=\varepsilon _{k,r}(1)-(r-k+2)=\theta(2,k,r,1)
$$
as wanted. 

As for (iii), suppose that for some $m>1$ one has $\theta_X(m)=\theta(2,k,r,m)$. Then the above argument implies that  
$\theta_Y(i)=\theta(1,k,r,i)$, for $1\leq i\leq m$. By Theorem \ref {thm:main}(ii) for $Y$, $Y$ is a rational normal curve of degree $r-k$ and therefore $X$ is a rational normal surface of degree $r-k$.
Conversely, if $X$ is a rational normal surface of degree $r-k$, then $X$ is projectively normal and we have equalities in \eqref {eq:inequ}, \eqref {eq:inequ2} and \eqref {eq:inequ3}. Moreover   $\theta_Y(i)=\theta(1,k,r,i)$ for all positive integers $i$ by Theorem \ref {thm:main}(ii) for $Y$, and therefore we have equalities in \eqref{eq:ineq}. This implies that  $\theta_X(m)=\theta(2,k,r,m)$.\end{proof}

The proof in the case of threefolds is similar so we will be brief.

\begin{proof}[Proof of part (a) of Theorem \ref {thm:main} for threefolds] Let $X$ be an irreducible projective Grassmann non--degenerate threefold  in $\mathbb G(k,r)$ not presenting the cone case and let $Y$ be a general hyperplane section of $X$, which, by Lemma \ref {lem:basic}, is Grassmann non--degenerate and does not present the cone case. Arguing as in the proof of the surface case, and applying part (a) of Theorem \ref {thm:main} for  $Y$, we have
\begin{align*}
\theta_X(m) &\leq \sum_{i=1}^m\theta_Y(i)-\sum_{i=1}^{m-1}\varepsilon _{k,r}(i)-1\leq \\
\nonumber  &\leq \sum_{i=1}^m \Big (\varepsilon _{k,r}(i)-\frac {i(i+1)}2 (r-k)-i-1 \Big)- \sum_{i=1}^{m-1}\varepsilon _{k,r}(i)-1=\\
\nonumber&=\varepsilon _{k,r}(m)- \Big(\sum_{i=1}^m\frac {i(i+1)}2\Big) (r-k)-\frac {m(m+1)}2-m-1=\theta(3,k,r,m)
\end{align*}
proving \eqref {eq:Gal3} in this case. The proof of (i) and (ii) proceed exactly in the same way as in the surface case, so we leave it to the reader.\end{proof}

We can now finish the proof of part (a) of Theorem \ref {thm:main}. The proof is similar to the surface and threefold case so we will again be brief.

\begin{proof}[Proof of part (a) of Theorem \ref {thm:main}, the general case] We will work by induction, since we have proved the theorem for $n=1,2,3$. Let $X$ be an irreducible projective Grassmann non--degenerate variety of dimension $n\geq 4$  in $\mathbb G(k,r)$ not presenting the cone case. Let $Y$ be its general hyperplane section, which by Lemma \ref {lem:basic},  is Grassmann non--degenerate and does not present the cone case. So we can apply induction on $Y$. Arguing as in the surface and threefold case and applying induction and \eqref {eq:triv}, we have
\begin{align*}
\theta_X(m) &\leq \sum_{i=1}^m\theta_Y(i)-\sum_{i=1}^{m-1}\varepsilon _{k,r}(i)-1\leq \\
\nonumber  &\leq \sum_{i=1}^m  (\varepsilon _{k,r}(i)-(r-k)\sigma(n-1,k,r,i)- \tau(n-1,k,r,i)) - \sum_{i=1}^{m-1}\varepsilon _{k,r}(i)-1=\\
\nonumber&=\varepsilon _{k,r}(m)- \Big(\sum_{i=1}^m\sigma(n-1,k,r,i) \Big) (r-k)- \Big(\sum_{i=1}^m
\tau(n-1,k,r,i)\Big)-1=\\
&=\theta(n,k,r,m)
\end{align*}
proving \eqref {eq:Gal3} in this case. The proof of (i) and (ii) proceeds exactly in the same way as in the surface case, so we leave it to the reader.
\end{proof}

\section{The cone case}\label{sec:cone}

Next we come to the proof of part (b) of Theorem \ref{thm:main}. For this we need a preliminary. 
Let $X$ be an irreducible, non degenerate variety of degree $d$ and dimension $n\geq 1$ in $\Proj^r$. We will denote by $X_ i$ the $i$--dimensional section of $X$ with a general subspace of dimension $r-n+i$ of $\Proj^r$. One has $X=X_n$. 
We note that 
\begin{equation}\label{eq:casteln}
h_{X_0}(m)\geq \min\{d,m(r-n)+1\}
\end{equation}
(see \cite   [Corollary (3.5)]{Ha}).

\begin{proposition}\label{prop:cast} Let $X$ be an irreducible, non degenerate variety of degree $d$ and dimension $n\geq 1$ in $\Proj^r$, with $r>n$. For all positive integers $m$, one has
\begin{equation}\label{eq:cast}
h_X(m)\geq {{m+n-1}\choose n}(r-n)+{{m+n}\choose n}.
\end{equation}
Moreover:\\
\begin{inparaenum}
\item[(i)]  equality holds in \eqref {eq:cast} for $m=1$;\\
\item[(ii)]  equality holds in \eqref {eq:cast}
for some $m\geq 2$ if and only if $X$ is a variety of minimal degree $d=r-n+1$.
\end{inparaenum}
\end{proposition}

\begin{proof} This proposition is essentially contained in Theorem 6.1 in \cite {CC}. However we give here a proof for completeness. 

First we note that trivially equality holds in  \eqref {eq:cast} for $m=1$, i.e., (i) holds. 
Next we prove \eqref {eq:cast} by double induction, first on the dimension $n$ and then on $m$. 

In the curve case, we proceed by induction on $m$, assuming $m\geq 2$. By \eqref {eq:cas} we have
$$
h_X(m)\geq h_X(m-1)+h_{X_0}(m)
$$
and by \eqref {eq:casteln}, by the induction and the fact that $d\geq r$ since $X$ is non--degenerate, we have 
\begin{equation}\label{soppo}
h_X(m)\geq (m-1)(r-1)+m+r=m(r-1)+m+1
\end{equation}
proving \eqref {eq:cast} in this case. If equality holds in \eqref {soppo} in particular we have 
$$
\min\{d,m(r-1)+1\}=r
$$
and since $m(r-1)+1\geq 2r-1>r$ (because $r>1$), we deduce $d=r$, hence $X$ is a rational normal curve. Conversely, if $X$ is a rational normal curve, one has $h_X(m)=mr+1$ and the equality holds in  \eqref {eq:cast}.

Next we treat the  case of varieties $X$ of dimension $n>1$  and we assume by induction that the proposition holds for varieties of dimension smaller than $n$. Moreover we proceed by induction on $m$.  By \eqref {eq:cas}, by induction and by \eqref {eq:casteln}, we have
\begin{align}\label{eq:pappa}
h_X(m)&\geq h_X(m-1)+h_{X_{n-1}}(m)\geq \\
\nonumber &\geq {{m+n-2}\choose n}(r-n)+{{m+n-1}\choose n}+\\
\nonumber &+{{m+n-2}\choose {n-1}}(r-n)+{{m+n-1}\choose {n-1}}=\\
\nonumber &={{m+n-1}\choose n}(r-n)+{{m+n}\choose n}
\end{align}
as wanted. Moreover, if the equality holds in \eqref {eq:pappa} for some $m>1$, then the equality holds in \eqref {eq:cast} for some $m>1$ for $X_{n-1}$. Then by induction $X_{n-1}$ is of minimal degree and therefore also $X$ is of minimal degree. Conversely, if $X$ is of minimal degree, then also $X_{n-1}$ is of minimal degree. Moreover $X$ is projectively normal. Because of this and by induction we have that equalities hold in \eqref {eq:pappa} and therefore equality holds in \eqref {eq:cast}. \end{proof}

Now we are in position to finish the:

\begin{proof}[Proof of part (b) of Theorem \ref{thm:main}] Let $X$ be an irreducible projective Grassmann non--dege\-nerate variety of dimension $n$  in $\mathbb G(k,r)$ presenting the cone case. Let $\Pi$ be the vertex of $Z(X)$, that is a $(k-1)$--dimensional subspace of $\Proj^r$. Then $X$ lies in the $(r-k)$--subspace $\Pi^\perp$ contained in $\mathbb G(k,r)$ , whose points correspond to linear spaces of dimension $k$ containing $\Pi$. Since each map $\rho_m$ as in \eqref {eq:rho} factors through the maps
$$
H^0(\G(k,r),\O_{\G(k,r)}(m))\longrightarrow  H^0(\Pi^\perp ,\O_{\Pi^\perp}(m)) \longrightarrow H^0(X,\O_{X}(m))
$$
and since 
$H^0(\G(k,r),\O_{\G(k,r)}(m))\longrightarrow  H^0(\Pi^\perp ,\O_{\Pi^\perp}(m))$
is clearly surjective, to compute $h_X(m)$ it suffices to compute it as a subvariety of $\Pi^\perp$. We can then apply Proposition \ref {prop:cast}. By \eqref {eq:sequ} we have
\begin{align*}
\theta_X(m) &=\varepsilon _{k,r}(m)-h_X(m) \leq \\
&\leq \varepsilon _{k,r}(m)-{{m+n-1}\choose n}(r-k-n)-{{m+n}\choose n}=\bar\theta(n,k,r,m)
\end{align*}
as wanted. If the equality holds for some $m\geq 2$, then in particular 
$$
h_X(m)={{m+n-1}\choose n}(r-k-n)+{{m+n}\choose n}
$$
and, by Proposition \ref {prop:cast}, $X$ is a variety of minimal degree in $\Pi^\perp$, proving the assertion. \end{proof}

\section{Examples}\label{sec:ex}

In the statement of Theorem \ref {thm:main} it is mentioned the possibility that  either 
$X$ does not present the cone case and is a rational normal 
variety of degree $r-k$ in a subspace of dimension $r-k+n-1$ in $\Proj^{N(k,r)}$, or $X$ presents the cone case and $X$ is a  rational normal 
variety of degree $r-k-n+1$ in a linear subspace of dimension $r-k+n-1$ contained in $\G(k,r)$. The latter case can clearly occur, whereas it is not a priori clear that the former case can occur for all values of $n\geq 2$ and $k$ (for $n=1$ also the former case occurs when $Z(X)$ is a rational normal $(k+1)$--dimensional variety of degree $r-k$ in $\Proj^r$). In this  section we provide two examples showing that for $n\geq 2$ also the former case can occur. 

\begin{example}\label{ex:uno}

Fix positive integers $n,k$ and set $r=n+k-1$. In $\Proj^r$, fix a $k$--linear space $\Pi$ and inside it a $(k-2)$--linear space $\pi$. Then let $X$ be the subvariety of $\G(k,r)$ described by all points corresponding to $k$--subspaces of $\Proj^r$ intersecting $\Pi$ in a subspace of dimension at least $k-1$ containing $\pi$. It is easy to check that $\dim(X)=n$, that it is Grassmann non--degenerate and that it does not present the cone case. Moreover $X$ is a cone with vertex the point corresponding to  $\Pi$. Let us  intersect $X$ with a hyperplane $H$ not containing $\Pi$. To do this, let us fix a general linear space $\Pi'$ of dimension $r-k-1=n-2$. Then $\Pi'$ does not intersect $\pi$ and we can consider the hyperplane section $H_{\Pi'}$ of $\G(k,n)$ consisting of all points corresponding to linear spaces of dimension $k$ intersecting $\Pi'$. The points in the intersection $X\cap H_{\Pi'}$ can be obtained in the following way. Take any $(k-1)$--space $P$ in $\Pi$ containing $\pi$, take a  point $p\in \Pi'$, take $\Pi_{P,p}:=\langle P,p\rangle$ and consider it as a point of $\G(k,r)$. One has
$$X\cap H_{\Pi'}=\bigcup _{\pi\subset P\subset \Pi, p\in \Pi'}\Pi_{P,p}.$$
Now the subspaces $P$ such that $\pi\subset P\subset \Pi$ vary in a $\Proj^1$, and $p\in \Pi'$ varies in a $\Proj^{n-2}$. An easy explicit computation shows that $X\cap H_{\Pi'}$ is isomorphic to the Segre variety ${\rm Seg}(1,n-2)\cong \Proj^1\times \Proj^{n-2}$ which has degree $n-1$ and spans a linear space of dimension $2n-3$. Hence $X$ has degree $n-1$ and spans a linear space of dimension $2n-2$. Thus $X$ is a rational normal variety of dimension $n$ and degree $n+1$ contained in $\G(k,n)$.
\end{example}

\begin{example}\label{ex:due}
Fix positive integers $n,k$ and set $r=k+4$. First consider the case $k=2h+1$ is odd. Then fix two skew subspaces $\Pi,\Pi'$ of dimension $h+2$ in $\Proj^r$ and inside $\Pi,\Pi'$  fix two $(h-1)$--subspaces $\pi, \pi'$ respectively. For any $h$--linear space $P$ [resp. $P'$] contained in $\Pi$ [resp. contained in $\Pi'$] and containing $\pi$ [resp. containing $\pi'$], consider $\Pi_{P,P'}=\langle P,P'\rangle$, that has dimension $2h+1=k$, so we can consider it as a point in $\G(k,r)$. The subspaces $P$ [resp. $P'$] contained in $\Pi$ [resp. contained in $\Pi'$] and containing $\pi$ [resp. containing $\pi'$] vary in a $\Proj^2$, hence
$$
V=\bigcup_{\pi\subset P\subset \Pi, \pi'\subset P'\subset \Pi'}\Pi_{P,P'}
$$
is a subvariety of $\G(k,r)$ which is easy to check to be isomorphic to the Segre variety ${\rm Seg}(2,2)\cong \Proj^2\times \Proj^2$ embedded in $\Proj^8$. As well known,  ${\rm Seg}(2,2)$ contains a Veronese surface $X\subset \Proj^5$, as the image of the diagonal of $\Proj^2\times \Proj^2$. The surface $X$ is a rational normal surface contained in $\G(k,n)$ which is Grassman non--degenerate and not presenting the cone case.

Consider now the case  $k=2h+2$ (and still $r=k+4$).  Now fix two skew subspaces $\Pi,\Pi'$ of dimension $h+2$ in $\Proj^r$ and a point $p$ off the hyperplane of $\Proj^r$ spanned by $\Pi$ and $\Pi'$. As in the odd case, inside $\Pi,\Pi'$ we fix two $(h-1)$--subspaces $\pi, \pi'$ respectively. For any $h$--linear space $P$ [resp. $P'$] contained in $\Pi$ [resp. contained in $\Pi'$] and containing $\pi$ [resp. containing $\pi'$], consider $\Pi_{P,P'}=\langle P,P',p\rangle$, that has dimension $k$, so we can consider it as a point in $\G(k,r)$. The subspaces $P$ [resp. $P'$] contained in $\Pi$ [resp. contained in $\Pi'$] and containing $\pi$ [resp. containing $\pi'$] vary in a $\Proj^2$, hence
$$
V=\bigcup_{\pi\subset P\subset \Pi, \pi'\subset P'\subset \Pi'}\Pi_{P,P'}
$$
is a subvariety of $\G(k,r)$ which is isomorphic to the Segre variety ${\rm Seg}(2,2)$. As in the odd case we see that $\G(k,r)$  contains a Veronese surface $X\subset \Proj^5$, that is a rational normal surface which is Grassman non--degenerate and not presenting the cone case.

\end{example}


\begin{thebibliography}{}

\bibitem {Ca} G. Castelnuovo, \emph{Sui multipli di una serie lineare di gruppi di punti appartenenti ad una curva algebrica}, Rend. del Circolo Mat. di Palermo, {\bf 7}, (1893), 89--110.

\bibitem{CC} L. Chiantini, C. Ciliberto, \emph{Weakly defective varieties}, Trans. Amer. Math. Soc., {\bf 354} (1), (2001), 151--178.

\bibitem {Gal1} D. Gallarati, \emph{Sul numero dei complessi algebrici di rette, di ordine assegnato, che contengono una data rigata algebrica},  Rend. Accad. Naz. Lincei, Ser. VIII {\bf 14} (2), (1953), 213--220.

\bibitem {Gal2} D. Gallarati, \emph{Sulle variet\`a di $S_r$ composte di $\infty^1S_k$, i cui $S_k$ appartengono al massimo numero di complessi lineari},  Rend. Accad. Naz. Lincei, Ser. VIII {\bf 14} (3), (1953),408--412.

\bibitem {Ha} J. Harris, \emph{Curves in projective space} (with the collaboration of D. Eisenbud), Les Presses de l'Universit\'e de Montr\'eal, 1982. 

\bibitem {HP} W.V.D. Hodge, D. Pedoe, \emph{Methods of algebraic geometry}, Vol. II, Cambridge University Press, 1952. 



 
\end{thebibliography}
\end{document}